\documentclass[11pt,reqno]{article}
\usepackage{latexsym, amsmath, amssymb, a4, epsfig}
\usepackage{amsthm}
\usepackage{graphicx}
\usepackage{latexsym, color}
\usepackage{caption}
\usepackage{subcaption}

\newtheorem{thm}{Theorem}[section]
\newtheorem{cor}[thm]{Corollary}
\newtheorem{lem}[thm]{Lemma}
\newtheorem{prop}[thm]{Proposition}

\usepackage[left=1in,right=1in,top=1in,bottom=1in]{geometry}

\newcommand{\mc}{\mathcal}

\newcommand{\DO}{\mathcal{D}_{\partial \Omega}}

\newcommand{\KO}{\mathcal{K}_{\partial \Omega}}

\renewcommand{\d}{\displaystyle}
\newcommand{\bef}{\begin{figure}}
	\newcommand{\enf}{\end{figure}}

\newcommand{\ds}{\displaystyle}
\newcommand{\pf}{\noindent {\sl Proof}. \ }
\newcommand{\p}{\partial}

\newcommand{\eqnref}[1]{(\ref {#1})}

\newcommand{\Cbb}{\mathbb{C}}

\newcommand{\Rbb}{\mathbb{R}}

\newcommand{\la}{\langle}
\newcommand{\ra}{\rangle}

\newcommand{\Kcal}{\mathcal{K}}

\newcommand{\Dcal}{\mathcal{D}}
\newcommand{\Pcal}{\mathcal{P}}

\newcommand{\Scal}{\mathcal{S}}

\def\Be{{\bf e}}

\newcommand{\Ga}{\alpha}
\newcommand{\Gb}{\beta}

\newcommand{\Gvf}{\varphi}
\newcommand{\Gg}{\gamma}

\newcommand{\Gl}{\lambda}

\newcommand{\Gt}{\theta}

\newcommand{\Gs}{\sigma}

\newcommand{\Go}{\omega}

\newcommand{\Gz}{\zeta}
\newcommand{\GD}{\Delta}

\newcommand{\GG}{\Gamma}

\newcommand{\GO}{\Omega}

\newcommand{\beq}{\begin{equation}}
\newcommand{\eeq}{\end{equation}}

\def\ol{\overline}
\numberwithin{equation}{section}
\numberwithin{figure}{section}

\author{Hyeonbae Kang\thanks{Department of Mathematics and Institute of Applied Mathematics, Inha University, Incheon
22212, S. Korea (hbkang@inha.ac.kr, xiaofeilee@hotmail.com).} \and
 Xiaofei Li\footnotemark[2] }
\begin{document}
\title{Construction of weakly neutral inclusions of general shape by imperfect interfaces\thanks{\footnotesize This work was
supported by NRF grants No. 2016R1A2B4011304 and 2017R1A4A1014735.}}
	
\maketitle

\begin{abstract}
Upon insertion of an inclusion into a medium with the uniform field, if the field is not perturbed at all outside the inclusion, then it is called a neutral inclusion. It is called a weakly neutral inclusion if the field is perturbed weakly. The inclusions neutral to multiple uniform fields are of circular shape if the medium is isotropic, and any other shape cannot be neutral. We consider in this paper the problem of constructing inclusions of general shape which are weakly neutral to multiple fields in two dimensions. We show that a simply connected domain satisfying a certain geometric condition can be realized as a weakly neutral inclusion to multiple fields by introducing an imperfect interface parameter on the boundary. The geometric condition on the domain and the imperfect interface parameter are determined by the first coefficient of the conformal mapping from the exterior of the unit disk onto the exterior of the domain. We provide some numerical examples to compare field perturbations by weakly neutral inclusions and perfectly bonding interfaces. They clearly show that the field perturbation by weakly neutral inclusions is much weaker.
\end{abstract}
	


%%%%%%%%%%%%%%%%%%%%%%%%%%%%%%%%%%%%%%%%%%%%%%%%%%%%%%%%%%%%%%%%%%%%%%%%%%%%%%%%%%%%%%%%%%%%%%%%%%%%%%%%%%%%%%%%
\section{Introduction}
%%%%%%%%%%%%%%%%%%%%%%%%%%%%%%%%%%%%%%%%%%%%%%%%%%%%%%%%%%%%%%%%%%%%%%%%%%%%%%%%%%%%%%%%%%%%%%%%%%%%%%%%%%%%%%%%

If an inclusion with a different material property is inserted into a medium with a uniform field, then the uniform field is perturbed in general. However, there are some inclusions which do not perturb the uniform field outside the inclusion, and such an inclusion is referred to as a neutral inclusion.

Neutrality of inclusions has a significant implication in the theory of composites. It is shown by Hashin \cite{H1} and Hashin-Shtrikman \cite{H2} that since insertion of neutral inclusions does not perturb the outside uniform field, the effective conductivity of the assemblage filled with such inclusions of many different scales is the same as the conductivity of the matrix without inclusions. It is also proved that such an effective conductivity is one of the Hashin-Shtrikman bounds on the effective conductivity of arbitrary two-phase composites. We refer to \cite{mbook} for developments on neutral inclusions in relation to the theory of composites.

Hashin's inclusions consists of the core and the shell with interfaces of perfect bonding. However, there is another way to make neutral inclusions: a single inclusion with an imperfect interface. The perfect bonding interface is characterized by the continuity of both flux and potential along the interface, while the imperfect interface is characterized by either discontinuity of the potential or discontinuity of the flux along the interface. The former one is referred to as the low conductivity (LC) type, while the latter as the high conductivity (HC) type \cite{B.M}.

Consider the following LC type imperfect interface problem for $d=2,3$:
\beq\label{LC}
\begin{cases}
	\nabla\cdot\Gs\nabla u=0 \quad &\mbox{in } \GO \cup (\Rbb^d\setminus\ol{\GO}) ,\\
	\Gb(u|_+ - u|_-)= 	\ds \Gs_m \frac{\p u}{\p\nu}\Big|_+ \quad&\mbox{on } \p \GO,\\
    \vspace{0.3em}
    \ds \Gs_c \frac{\p u}{\p\nu}\Big|_- = \Gs_m \frac{\p u}{\p\nu}\Big|_+ \quad&\mbox{on } \p \GO,\\
	\vspace{0.3em}
	u(x)-a \cdot x =O(|x|^{1-d}) \quad&\mbox{as }|x|\rightarrow\infty,
\end{cases}
\eeq
where $a$ is a unit vector representing the background uniform field, the subscripts $+$ and $-$ indicate the limits from outside and inside $\GO$ to $\p \GO$, respectively, $\Gb$ is the interface parameter of LC type, which is a non-negative function defined on the interface $\p \GO$, and $\Gs$ is a piecewise constant function defined by
\begin{equation*}
\Gs=\Gs_c\chi(\GO)+\Gs_m\chi(\Rbb^d\setminus \ol{\GO}).
\end{equation*}
Here, $\chi$ denotes the characteristic function. The conductivity $\Gs_c$ is assumed to be isotropic (scalar) and constant, but $\Gs_m$ is allowed to be anisotropic, i.e., a positive definite symmetric constant matrix.

It is proved in \cite{T.R} (see also \cite{B,L,L.V2,L.V}) that if $\GO$ is a disk (or a ball) of radius $r$ and
\beq\label{Gbdisk}
\Gb=\frac{1}{r}\frac{\Gs_c\Gs_m}{\Gs_c-\Gs_m},
\eeq
then the solution $u$ to \eqnref{LC} satisfies
\beq\label{neutral2}
u(x)-a \cdot x \equiv 0 \quad\mbox{for all } x \in \Rbb^d\setminus \GO,
\eeq
in other words, the uniform field $-\nabla(a\cdot x)$ is unperturbed outside $\GO$. We also refer to \cite{B.M,M.B,Ru} for neutral inclusions of general shape with respect to a single uniform field by imperfect interfaces.

The neutral inclusion presented above is of circular (spherical) shape. Actually this is the only shape for the neutral inclusions (to multiple uniform fields). In fact, we prove in this paper (Theorem \ref{thm:ellipsoid} in Appendix) that the only neutral inclusions are disks (balls) with constant interface parameters if $\Gs_m$ is isotropic, and they are ellipses (ellipsoids) for the anisotropic case.

In this paper we consider weakly neutral inclusions of general shape.
The neutrality requires $u(x)-a \cdot x \equiv 0$ outside $\GO$, while the condition at $\infty$ in \eqnref{LC} requires $u(x)-a \cdot x =O(|x|^{1-d})$ as $|x| \to \infty$. In other words, the perturbation is of order $|x|^{1-d}$ for general inclusions, but it is completely vanishing for neural inclusions. The weakly neutral inclusions in between them, namely, they are inclusions such that the solution $u$ to  \eqnref{LC} satisfies
\beq\label{weakneucond}
u(x)-a \cdot x =O(|x|^{-d}), \quad\mbox{as } |x| \to \infty.
\eeq

The leading order term in the far-field expansion of the solution to \eqnref{LC} is expressed in terms of the polarization tensor (PT) associate with the inclusion $\GO$ and the interface parameter $\Gb$ (see \eqnref{faf-field}). Thus the weakly neutral inclusion is in fact the PT-vanishing structure. The notion of the PT-vanishing structure, or more generally, that of higher order PT-vanishing structure was introduced in \cite{AKLL}. Such structures can be realized by multilayered circular structures and were used in an essential way for dramatic enhancement of near cloaking. Like the neutral inclusion, weakly neutral inclusions are related to effective properties of composites. In fact, the leading order term in the asymptotic expansion of the effective properties of two phase dilute composites, as the volume fraction tends to zero, is given by the PT (see \cite{AKT05} and references therein). Thus, if the inclusions are weakly neutral, then the leading order term vanishes.

We emphasize that a simply connected inclusion (with the constant conductivity) cannot be weakly neutral if the interface is perfectly bonding. In fact, it is shown \cite{AKBook1} that if $u$ is the solution to the perfectly bonding problem, then the following asymptotic expansion holds:
\beq\label{PTexpansion}
u(x)-a\cdot x = \frac{1}{\Go_d} \frac{\la Ma, x \ra}{|x|^d} + O(|x|^{-d}), \quad |x| \to \infty,
\eeq
where $\Go_d$ is the surface area of the unit sphere and $M$ the polarization tensor of the problem, which is determined by $\GO$ (and $\Gs_c$, $\Gs_m$). If $\GO$ is simply connected, then $M$ is positive- or negative-definite depending on the sign of $\Gs_c-\Gs_m$. Thus there is a constant $C$ independent of the unit vector $a$ such that
\beq\label{lower}
|u(x)-a \cdot x| \ge C|x|^{1-d}
\eeq
as $|x| \to \infty$ in some direction determined by $a$. In particular, the condition \eqnref{weakneucond} cannot be satisfied.

To construct weakly neutral inclusions of general shape using imperfect interfaces, we use the conformal transformation to pull back the problem to the unit disk. For this reason we assume the conductivity $\Gs_c$ of the inclusion is infinity. We also assume that $\Gs_m=1$, which is just for simplicity.
So, the problem \eqnref{LC} becomes
\beq\label{main}
\begin{cases}
	\Delta u=0 \quad &\mbox{in }\Rbb^2\setminus\ol{\GO},\\
	\Gb(u-\Gl)= 	\ds\frac{\p u}{\p\nu}\Big|_+ \quad&\mbox{on }\p \GO,\\
	\vspace{0.3em}
	u(x)-a \cdot x=O(|x|^{-1}) \quad&\mbox{as }|x|\rightarrow\infty,
\end{cases}
\eeq
where $\Gl$ is a constant which is uniquely determined by the condition $\int_{\p \GO}\frac{\p u}{\p \nu}\, ds=0$, i.e.,
\begin{equation*}
\Gl = \frac{\int_{\p \GO} \Gb u ds}{\int_{\p \GO} \Gb ds}.
\end{equation*}
The purpose of this paper is to construct $\Gb$ on $\p\GO$ for general $\GO$ so that the solution $u$ to \eqnref{main} satisfies \eqnref{weakneucond}.

If $\GO$ is simply connected, there is a conformal mapping from the exterior of the unit disk to the exterior of $\GO$ of the form
$$
\Phi(\Gz)=b_{-1}\Gz+b_0+\frac{b_1}{\Gz}+\dots.
$$
By dilating and translating $\GO$ if necessary, we assume that $b_{-1}=1$ and $b_0=0$. Then let
\begin{equation*}
b_{\GO}:= b_1 \quad\mbox{and}\quad \Phi_\GO:=\Phi,
\end{equation*}
so that $\Phi_\GO$ takes the form
\beq\label{conformal}
\Phi_\GO(\Gz)=\Gz + \frac{b_\GO}{\Gz}+\dots, \quad |\Gz|>1. 
\eeq
We emphasize that
\beq\label{bGOone}
|b_\GO| <1.
\eeq
In fact, we have $|b_\GO| \le 1$ as a consequence of area theorem (or the Bieberbach conjecture). However, in the extreme case when $|b_\GO|=1$, the conformal mapping is of the form $\Gz+b_\GO \Gz^{-1}$, which maps the exterior of the unit disk to the exterior of a slit. So such cases are excluded.  Thus we have \eqnref{bGOone}.

The following is the main result of this paper.

\begin{thm}\label{main_thm}
Let $\GO$ be a bounded simply connected domain in $\Rbb^2$ with the Lipschitz boundary such that the conformal mapping $\Phi_\GO$ takes the form \eqnref{conformal}. Assume that
\beq\label{assumebGO}
|b_\GO|\leq 2-\sqrt{3}.
\eeq
Define $\Gb$ on $\p\GO$ by
\beq\label{Gbcon}
\Gb(\Phi_\GO(e^{i\Gt})) = \left(\frac{1}{1+|b_\GO|} + \frac{1}{1-|b_\GO|} -1 + \left(\frac{2}{1+|b_\GO|} - \frac{2}{1-|b_\GO|}\right) \cos2\Gt\right) \frac{1}{|\Phi_\GO'(e^{i\Gt})|}.
\eeq
Then $(\GO,\Gb)$ is weakly neutral to all uniform fields, namely, the solution $u$ to \eqnref{main} satisfies \eqnref{weakneucond} for all $a$ with $|a|=1$.
\end{thm}

The function $\Gb$ is well defined for any simply connected domain $\GO$ because of \eqnref{bGOone}. However, the assumption \eqnref{assumebGO} is imposed to guarantee that the function $\Gb$ is non-negative, which we need to ensure that \eqnref{main} has a unique solution. We assume in the theorem that the conformal mapping $\Phi_\GO$ takes the form \eqnref{conformal}. But this is not a restriction. In fact, once we construct $\Gb$ under this assumption, then one can construct $\Gb$ for a general $\GO$ using the translation and rotation formula (Lemma \ref{transrot}).

We also provide some numerical examples of solutions to \eqnref{main} with $\Gb$ in \eqnref{Gbcon} and compare them with solutions with perfectly bonding interfaces. They clearly demonstrate that weak neutrality \eqnref{weakneucond} is achieved.

This paper is organized as follows. In section 2, we present the solution to \eqref{main} using layer potential techniques. In section 3, we construct weakly neutral inclusions using conformal transformation.  Numerical results are given in section 4. This paper ends with a short discussion. Appendix is to prove that ellipses (ellipsoids) are the only shape for neutral inclusions when the background conductivity is anisotropic.

%%%%%%%%%%%%%%%%%%%%%%%%%%%%%%%%%%%%%%%%%%%%%%%%%%%%%%%%%%%%%%%%%%%%%%%%%%%%%%%%%%%%%%%%%%%%%%%%%%%%%%%%%%%%%%%%%%%%
\section{Representation of the solution}
%%%%%%%%%%%%%%%%%%%%%%%%%%%%%%%%%%%%%%%%%%%%%%%%%%%%%%%%%%%%%%%%%%%%%%%%%%%%%%%%%%%%%%%%%%%%%%%%%%%%%%%%%%%%%%%%%%%%

In this section we obtain a representation formula of the solution to \eqnref{main} from which one can see that the leading order term of the asymptotic expansion at infinity is described by the polarization tensor. Some of the arguments of this section are from \cite{Li}.

The representation is expressed in terms of layer potentials. Let
\begin{equation*}
\GG(x)=\frac{1}{2\pi} \ln |x|,
\end{equation*}
the fundamental solution to the Laplacian in two dimensions. The single and double layer potentials of a function $\Gvf$ on $\p \GO$ are defined to be
\begin{align*}
\Scal_{\p \GO}[\Gvf](x) &:= \int_{\p \GO} \GG(x-y) \Gvf(y)\, ds(y), \quad x\in \Rbb^2, \\
\DO[\psi](x) &:= \int_{\p \GO} \p_{\nu_y}\GG(x-y) \Gvf(y)\, ds(y), \quad x\in \Rbb^2 \setminus \p \GO,
\end{align*}
where $\p_{\nu_y}$ denotes outward normal derivative with respect to $y$-variables. It is well known (see, for example, \cite{AmKa07Book2}) that the single and double layer potentials satisfy the following jump relations:
\begin{align}
\p_\nu \Scal_{\p \GO}[\Gvf](x)\Big|_{\pm} &=(\pm \ds\frac{1}{2}I+\Kcal_{\p \GO}^{*})[\Gvf](x), \quad\mbox{a.e. } x\in\p \GO, \label{singlejump} \\
\Dcal_{\p \GO}[\Gvf](x)\big|_{\pm} &=(\mp\ds\frac{1}{2}I+\Kcal_{\p \GO})[\Gvf](x), \quad\mbox{a.e. } x\in\p \GO,
\label{doublejump}
\end{align}
where the operator $\KO$ on $\p \GO$ is defined by
\begin{equation*}
\Kcal_{\p \GO}[\Gvf](x)= \mbox{p.v.} \int_{\p \GO} \p_{\nu_y}\GG(x-y)  \Gvf(y)\, ds(y),
\end{equation*}
and $\Kcal_{\p \GO}^*$ is the $L^2$-adjoint of $\Kcal_{\p\GO}$. Here p.v. stands for the Cauchy principal value.

Let $M_\Gb$ be the multiplication operator by $\Gb$. Then the solution $u$ to \eqnref{main} can be represented as
\beq\label{solrep1}
u(x)= a \cdot x - \left( \Scal_{\p \GO} M_\Gb -\Dcal_{\p \GO} \right ) [\psi](x), \quad x\in\Rbb^2\setminus \p \GO,
\eeq
for some function $\psi \in H^{1/2}(\p \GO)$. In fact, if $\psi$ satisfies the integral equation
\beq\label{intPcal}
\Pcal_{\p \GO}[\psi]= - a \cdot \nu \quad\mbox{on } \p \GO,
\eeq
where the operator $\Pcal_{\p \GO}$ is defined by
\begin{equation*}
\Pcal_{\p \GO}: = \left( \frac{1}{2}I-\Kcal_{\p \GO}^{*}\right)M_{\Gb} + \p_\nu\DO,
\end{equation*}
one can easily see using jump relations \eqnref{singlejump} and \eqnref{doublejump} that the function $u$ defined by \eqnref{solrep1} satisfies the interface conditions (the second line in \eqnref{main}). Moreover, if in addition $\psi$ satisfies
\beq\label{Gbzero}
\int_{\p\GO} \Gb \psi ds=0,
\eeq
then $u$ satisfies the decay condition at $\infty$ (the last line in \eqnref{main}). Thus $u$ is the solution to \eqnref{main}.

Let us now discuss about solvability of the integral equation \eqnref{intPcal}. For ease of notation we let
$$
X:= H^{-1/2}(\p \GO), \quad X_0:= H_0^{-1/2}(\p \GO), \quad Y:= H^{1/2}(\p \GO), \quad Y_0:= H^{1/2}_0(\p \GO).
$$
We suppose that $\Gb$ is smooth enough so that $M_\Gb$ maps $Y$ into $Y$. Let
$$
Y_\Gb := \{ \psi \in Y : \psi \mbox{ satisfies \eqnref{Gbzero}} \}.
$$
Since $\Kcal_{\p \GO}^{*}$ maps $X_0$ into $X_0$, and $\p_\nu\DO$ maps $Y$ into $X_0$, we infer that $\Pcal_{\p \GO}$ maps $Y_\Gb$ into $X_0$.

\begin{prop}\label{prop}
The operator $\Pcal_{\p \GO}: Y_\Gb \to X_0$ is invertible.
\end{prop}

\pf
Let us first observe that $\p_\nu\DO$ maps $Y$ into $X$, and $\Pcal_{\p \GO} - \p_\nu\DO= ( 1/2I-\Kcal_{\p \GO}^{*})M_{\Gb}$ maps $Y$ into $Y$. Since the embedding $Y\hookrightarrow X$ is compact, we may view $\Pcal_{\p \GO}: Y_\Gb \to X_0$ as a compact perturbation of $\p_\nu\DO$.

To prove that $\p_\nu\DO$ is invertible from $Y_\Gb$ onto $X_0$, let $g \in X_0$ and let $u$ be the unique solution in $H^1(\GO)$, under the normalization $\int_{\p \GO} u =0$, of $\GD u=0$ in $\p \GO$ satisfying $\p_\nu u =g$ on $\p \GO$. Define $f \in Y$ by
$$
f= \Big( \frac{1}{2} I + \Kcal_{\p \GO} \Big)^{-1} [u|_{\p \GO}],
$$
so that $\Dcal_{\p \GO}[f] = u|_{\p \GO}$. Note that since $\int_{\p \GO} u ds =0$, we have in fact $f \in Y_0$ and $\p_\nu\DO[f]=g$. Now we define
$$
\psi := f - \frac{\int_{\p \GO} \Gb f ds}{\int_{\p \GO} \Gb ds},
$$
so that $\psi \in Y_\Gb$. Since $\p_\nu\DO[1]=0$, we have $\p_\nu\DO[\psi]=g$ as desired.

We now show that $\Pcal_{\p \GO}$ is invertible from $Y_\Gb$ onto $X_0$. To do so it suffices to show that $\Pcal_{\p \GO}$ is injective since it is a compact perturbation of an invertible operator $\p_\nu\DO$. Suppose that
\beq\label{Pcalinjec}
\Pcal_{\p \GO}[\psi]= \left( \left( \frac{1}{2}I-\Kcal_{\p \GO}^{*}\right)M_{\Gb} + \p_\nu\DO \right)[\psi] = 0.
\eeq
Then $\psi$ is the solution to the system of integral equations \eqnref{intPcal} with the zero right-hand side. It means that $u$ defined by
$$
u(x)=-\left(\Scal_{\p \GO}M_\Gb-\Dcal_{\p \GO}\right)[\psi](x), \quad x \in \Rbb^2\setminus \GO,
$$
is the solution in $H^1(\Rbb^2 \setminus \GO)$ to \eqnref{main} with $a=0$. Then we have
\begin{align*}
0 &= \int_{\Rbb^2 \setminus \GO} |\nabla u|^2 dx + \int_{\p \GO} u \frac{\p u}{\p\nu}\Big|_+ ds \\
& = \int_{\Rbb^2 \setminus \GO} |\nabla u|^2 dx + \int_{\p \GO} \frac{1}{\Gb} \left( \frac{\p u}{\p\nu}\Big|_+ \right)^2 ds ,
\end{align*}
where the second equality holds since $\int_{\p \GO} \frac{\p u}{\p\nu}|_+ ds =0$. It then follows that $u=0$ in $\Rbb^2 \setminus \GO$ because $\Gb \geq 0$, namely,
$$
-\left(\Scal_{\p \GO}M_\Gb-\Dcal_{\p \GO}\right)[\psi](x)=0, \quad x \in \Rbb^2\setminus \GO.
$$
By taking the normal derivative, we obtain from \eqnref{singlejump} that
$$
\left( \left( \frac{1}{2}I+\Kcal_{\p \GO}^{*}\right)M_{\Gb} - \frac{\p}{\p\nu}\DO \right)[\psi]=0 \quad\mbox{on } \p \GO.
$$
This together with \eqnref{Pcalinjec} yields that $\psi=0$. This completes the proof.
\qed

We obtain the following corollary.
\begin{cor}
The problem \eqnref{main} has a unique solution in $H^1(\Rbb^2 \setminus \GO)$. The solution can be represented as \eqnref{solrep1}.
\end{cor}

Since the following expansion holds for $y \in \p \GO$ and $|x| \to \infty$:
$$
\GG(x-y)= \GG(x) - \nabla \GG(x) \cdot y + O(|x|^{-2}),
$$
and further $\int_{\p \GO} \Gb \psi ds=0$, we obtain from \eqnref{solrep1} that
\begin{equation*}
u(x)= a \cdot x + \int_{\p \GO} (\nabla \Gamma(x) \cdot y M_\Gb - \nabla \Gamma(x) \cdot \nu_y ) \psi(y) \, ds(y) + O(|x|^{-2}), \quad\mbox{as } |x| \to \infty.
\end{equation*}
Let, for $j=1,2$, $\Gvf_j$ be the solution to
\beq\label{intPcalnui}
\Pcal_{\p \GO}[\Gvf_j]= - \nu_j \quad\mbox{on } \p \GO,
\eeq
where $\nu_j$ is the $j$-th component of $\nu$. Then, $\psi=a_1 \Gvf_1 + a_2 \Gvf_2$, and we have
$$
u(x)= a \cdot x + \int_{\p \GO} (\nabla \Gamma(x) \cdot y M_\Gb - \nabla \Gamma(x) \cdot \nu_y ) (a_1 \Gvf_1 + a_2 \Gvf_2)(y) \, ds(y) + O(|x|^{-2}),
$$
which can be written as
\beq\label{faf-field}
u(x)= a \cdot x + \frac{1}{2\pi} \frac{\la Ta, x \ra}{|x|^2} + O(|x|^{-2}), \quad\mbox{as } |x| \to \infty.
\eeq
Here, $T=T(\GO, \Gb)=(T_{ij})_{i,j=1}^2$ is defined by
\beq\label{Tdef}
T_{ij}(\GO,\Gb):=\int_{\p \GO}(y_i M_{\Gb}- \nu_i (y))\Gvf_j(y)\, ds(y),
\eeq
which is the polarization tensor (abbreviated by PT) for the problem \eqnref{main}.

One can see easily by simple changes of variables that the PT enjoys the following properties.
\begin{lem}\label{transrot}
\begin{itemize}
\item[(i)]	The polarization tensor \eqref{Tdef} is invariant under translation.
\item[(ii)]	Let $\GO'$ be a domain and $\GO=\mc{R}\GO'$ where $\mc{R}$ is a rotation. Then
	$$
	T(\GO,\Gb)=\mc{R}T(\GO',\Gb\circ\mc{R})\mc{R}^\top,
	$$
	where $\mc{R}^\top$ denotes the transpose of $\mc{R}$.
\end{itemize}
\end{lem}

%%%%%%%%%%%%%%%%%%%%%%%%%%%%%%%%%%%%%%%%%%%%%%%%%%%%%%%%%%%%%%%%%%%%%%%%%%%%%%%%%%%%%
\section{Construction of weakly neutral inclusions}
%%%%%%%%%%%%%%%%%%%%%%%%%%%%%%%%%%%%%%%%%%%%%%%%%%%%%%%%%%%%%%%%%%%%%%%%%%%%%%%%%%%%%

Let $u$ be the solution to \eqref{main}. Since $u(x)-a\cdot x$ tends to $0$ as $|x| \to \infty$ and $\GO$ is simply connected, there is a function $U$ analytic in $\Cbb\setminus\overline{\GO}$ such that
$$
U'=u.
$$
Here and throughout this paper $U'$ and $U''$ respectively stand for the real and imaginary parts of $U$, namely, $U=U'+ i U''$.
By Cauchy-Riemann equations, one can see that the second line of \eqref{main} reads
\beq
\Gb\left(\frac{1}{2}(U+\ol{U})-\Gl\right)=\frac{1}{2i}\frac{d(U-\ol{U})}{ds} \quad \mbox{on }\p \GO,
\label{trans}
\eeq
where $\frac{d}{ds}$ is the tangential derivative. Since $U$ is analytic near $\infty$, it admits the following expansion:
$$
U(z)=\Ga z+\frac{\Ga_1(\Ga)}{z}+\frac{\Ga_2(\Ga)}{z^2}+\dots,
$$
where $\Ga=a_1-ia_2$ ($a=(a_1, a_2)$ is the uniform field appearing in \eqnref{main}). The weakly neutral condition \eqref{weakneucond} is equivalent to
\beq\label{a1}
\Ga_1(\Ga)=0 \quad \mbox{for  all }  \Ga~(\mbox{or equivalently for } \Ga=1,i).
\eeq

Suppose that the conformal mapping $\Phi_\GO$ is of the form \eqnref{conformal}.
Let $V^\Ga=U\circ\Phi_\GO$.
Then we have
\begin{equation*}
\begin{split}
V^\Ga(\Gz)&=\Ga\Phi_\GO(\Gz)+\frac{\Ga_1(\Ga)}{\Phi_\GO(\Gz)}+\frac{\Ga_2(\Ga)}{\Phi_\GO(\Gz)^2}+\dots\\
&=\Ga\Gz+\frac{\Ga b_\GO+\Ga_1(\Ga)}{\Gz}+\dots.
\end{split}
\end{equation*}
The transmission condition \eqref{trans} is transformed by $\Phi_\GO$ to
\begin{equation*}
\Gb(\Phi_\GO(\Gz))\left(\frac{V^\Ga+\ol{V^\Ga}}{2}-\Gl\right)=\frac{1}{2i}\frac{d(V^\Ga-\ol{V^\Ga})}{ds}\frac{1}{\left|\Phi_\GO'(\Gz)\right|} \quad \mbox{on } \p \GD,
\end{equation*}	
where $\Delta$ denotes the unit disk. Let
\beq\label{beta}
\Gg(\Gz)=\Gb(\Phi_\GO(\Gz))\left|\Phi_\GO'(\Gz)\right|,\quad |\Gz|=1.
\eeq
Then $\Gg$ is the imperfect interface parameter in $\Gz$-plane.

Note that \eqref{a1} is fulfilled if and only if $V^\Ga$ takes the form
\beq\label{VGaexp}
V^\Ga(\Gz)=\Ga\Gz+\frac{\Ga b_\GO}{\Gz}+\dots \quad \mbox{as } |\Gz|\rightarrow \infty,
\eeq
for $\Ga=1, i$. Let $v^\Ga := (V^\Ga)'$ for $\Ga=1, i$. Then $v^\Ga$ is the solution to \eqnref{main} with $\GO$ replaced by $\GD$ and $\Gb$ replaced by $\Gg$. The uniform field is given by $a=(1,0)$ if $\Ga =1$ and $a=(0,1)$ if $\Ga=i$.
The expansion at infinity takes the form
$$
v^\Ga(x)= x_1 + \frac{b'_\GO x_1 + b''_\GO x_2}{|x|^2} + \mbox{higher order terms}
$$
if $\Ga=1$, and
$$
v^\Ga(x)= x_2 + \frac{b''_\GO x_1 - b'_\GO x_2}{|x|^2} + \mbox{higher order terms}
$$
if $\Ga=i$.

In view of \eqnref{faf-field} we infer that in order for \eqnref{a1} to be satisfied, the polarization tensor $T(\GD,\Gg)$ needs to be of the form
\beq\label{Tb}
T(\GD,\Gg)=2\pi
\begin{bmatrix}
b'_\GO& b''_\GO \\
b''_\GO  & -b'_\GO\\
\end{bmatrix} .
\eeq
Note that the eigenvalues of the above matrix are $\pm|b_\GO|$. So after rotation if necessary it is sufficient to have
\beq\label{Tgamma}
T(\GD,\Gg)=2\pi
\begin{bmatrix}
|b_\GO|& 0\\
0  & -|b_\GO|\\
\end{bmatrix} .
\end{equation}

Then the problem of finding an imperfect interface parameter $\Gb$ such that \eqref{weakneucond} holds in $z$-plane is transformed to the problem of finding an imperfect interface parameter $\gamma$ such that \eqref{Tgamma} holds in $\Gz$-plane.

We prove the following proposition.
\begin{prop}\label{main_result}
Let $\GD$ be the unit disk. For any complex number $b_\GO$ with $|b_\GO|\leq 2-\sqrt{3}$, define $\Gg$ on $\p\GD$ by
\beq\label{Ggdisk}
\Gg(e^{i\Gt})=\frac{1}{1+|b_\GO|}+\frac{1}{1-|b_\GO|}-1+\left(\frac{2}{1+|b_\GO|}-\frac{2}{1-|b_\GO|}\right)\cos2\Gt.
\eeq
Then $\Gg \ge 0$ and $(\Delta,\Gg)$ has the polarization tensor of the form of \eqref{Tgamma}.
\end{prop}
Theorem \ref{main_thm} immediately follows from Proposition \ref{main_result}.

The rest of this section is devoted to the proof of Proposition \ref{main_result}. We first derive a general formula for the PT $T(\GD,\Gg)=(T_{ij})$. Recall from \eqref{Tdef} that
\beq\label{Tdefdisk}
T_{ij}=\int_{\p \GD}(y_i M_{\Gg} -\nu_i (y))\Gvf_j(y)\, ds(y),
\eeq
where $\Gvf_j$ is the solution to \eqref{intPcalnui}, namely,
\beq\label{G1}
\left( (\frac{1}{2} I - \Kcal_{\p \Delta}^{*}) M_\Gg + \p_\nu\mc{D}_{\p\Delta} \right) \Gvf_j =-\nu_j.
\eeq

Functions $\Gg$ and $\Gvf_j$, $j=1,2$, have the following Fourier expansions:
\beq\label{Fgamma}
\Gg=\sum_{n=-\infty}^{+\infty} \Gg_{n} e^{in\Gt}, \quad
\Gvf_j=\sum_{n=-\infty}^{+\infty} \Gvf_{j,n} e^{in\Gt}.
\eeq
Since $\Gg$ and $\Gvf_j$, $j=1,2$, are real functions on $\p \GD$, we  have
\beq
\Gg_{-n}=\ol{\Gg_{n}}\quad\mbox{and} \quad\Gvf_{j,-n}=\ol{\Gvf_{j,n}} \quad \mbox{for }n=1,\dots.
\label{conj}
\eeq
The condition \eqref{Gbzero} reads
\beq\label{L20}
\int_{\p \Delta}\Gg\Gvf_j\, ds=2\pi \sum_{k=-\infty}^{+\infty} \Gg_k \Gvf_{j,-k}=0, \quad\mbox{for } j=1,2.
\eeq

It is well known that
$$
\Kcal_{\p \GD}^{*}[e^{in\Gt}]=
\begin{cases}
0 \quad \mbox{if } n\neq 0, \\
\d\frac{1}{2} \quad \mbox{if } n=0,
\end{cases}
$$
and
$$
\p_\nu\mc{D}_{\p \GD}[e^{in\Gt}]=\frac{|n|}{2}e^{in\Gt} , \quad n=0,\pm 1,\pm2,\dots.
$$
See, for example, \cite{ACKLM} for proofs. With the help of these equalities, the integral equation \eqref{G1} becomes
\begin{equation}
\begin{cases}
\ds \sum_{n=-\infty}^{+\infty} \sum_{k=-\infty}^{+\infty} \Gg_k \Gvf_{1,n-k}e^{in\Gt} + \sum_{n=-\infty}^{+\infty} |n|\Gvf_{1,n}e^{in\Gt} = -e^{i\Gt}-e^{-i\Gt},\\
\ds \sum_{n=-\infty}^{+\infty} \sum_{k=-\infty}^{+\infty} \Gg_k \Gvf_{2,n-k}e^{in\Gt} + \sum_{n=-\infty}^{+\infty} |n|\Gvf_{2,n}e^{in\Gt} = ie^{i\Gt}-ie^{-i\Gt}.\\
\end{cases}
\label{int}
\end{equation}

Multiplying both sides  of the first equation of \eqref{int} by $e^{in\Gt},n=0,\pm1,\dots$, and integrating from $\theta=0$ to $\theta=2\pi$ we have the following system of $\Gg$ and $\Gvf_1$:
\begin{equation}\label{gammaphi1}
\begin{cases}
\sum \Gg_k \Gvf_{1,-1-k} + \Gvf_{1,-1}
=-1,\\
\sum \Gg_k \Gvf_{1,1-k} + \Gvf_{1,1}
=-1,\\
\sum \Gg_k \Gvf_{1,n-k} + |n|\Gvf_{1,n}=0\quad\mbox{if } n\neq \pm 1,\\
\end{cases}
\end{equation}
where the summation is over $k$ from $-\infty$ to $\infty$.
Considering the real part and the imaginary part of the \eqref{gammaphi1} separately, we have the following system
\begin{equation}
\begin{cases}
\sum\Gg_k \Gvf'_{1,-1-k}+\Gvf'_{1,1} =-1,\\
\sum\Gg_k  \Gvf''_{1,-1-k} - \Gvf''_{1,1} =0,\\
\sum\Gg_k  \Gvf'_{1,1-k} + \Gvf'_{1,1}
=-1,\\
\sum\Gg_k  \Gvf''_{1,1-k} + \Gvf''_{1,1}
=0,\\
\sum\Gg_k  \Gvf'_{1,n-k} +|n| \Gvf'_{1,n} =0 \quad \mbox{if } n\neq \pm 1,\\
\sum \Gg_k \Gvf''_{1,n-k} +|n|\Gvf''_{1,n}=0 \quad \mbox{if } n\neq \pm 1.
\end{cases}
\label{phi1}
\end{equation}

In the same way, we obtain the following system of equations from the second equation of \eqref{int}:
\begin{equation}
\begin{cases}
\sum\Gg_k \Gvf'_{2,-1-k}+\Gvf'_{2,1}
=0,\\
\sum\Gg_k \Gvf''_{2,-1-k}-\Gvf''_{2,1}
=-1,\\
\sum\Gg_k \Gvf'_{2,1-k}+\Gvf'_{2,1}
=0,\\
\sum\Gg_k \Gvf''_{2,1-k}+ \Gvf''_{2,1}
=1,\\
\sum\Gg_k \Gvf'_{2,n-k} +|n| \Gvf'_{2,n} =0 \quad \mbox{if } n\neq  \pm 1,\\
\sum\Gg_k \Gvf''_{2,n-k}  +|n| \Gvf''_{2,n} =0 \quad \mbox{if } n\neq  \pm 1.\\
\end{cases}
\label{phi2}
\end{equation}

The polarization tensor $T(\GD, \gamma)=(T_{ij})_{i,j=1}^2$ given by \eqnref{Tdefdisk} is now summarized as follows:
\begin{equation}
\begin{cases}
T_{11}&=\ds  \pi\sum_{k=-\infty}^{+\infty} \Gg_k(\Gvf_{1,-k-1}+\Gvf_{1,1-k})-2\pi \Gvf'_{1,1} ,\\
T_{12}&=\ds  \pi\sum_{k=-\infty}^{+\infty} \Gg_k(\Gvf_{2,-k-1}+\Gvf_{2,1-k})-2\pi \Gvf'_{2,1} ,\\
T_{21}&=\ds  -i\pi\sum_{k=-\infty}^{+\infty} \Gg_k(\Gvf_{1,-k-1}-\Gvf_{1,1-k})+2\pi \Gvf''_{1,1} ,\\
T_{22}&=\ds  -i\pi\sum_{k=-\infty}^{+\infty} \Gg_k(\Gvf_{2,-k-1}-\Gvf_{2,1-k})+2\pi \Gvf''_{2,1} ,
\end{cases}
\label{Tij}
\end{equation}
where $\Gg_k$ and $\Gvf_{j,k}$, $j=1,2$, $k=0,\pm 1,\pm 2,\dots$ satisfy \eqref{phi1} and \eqref{phi2}.

So far, we derive a general formula for the polarization tensor for the disk with the imperfect interface. We now prove Proposition \ref{main_result}. In view of \eqnref{Ggdisk}, we seek $\Gg$ in the following form:
\beq
\Gg=\Gg_{0}+2\Gg_2\cos 2\Gt.
\label{gamma2}
\eeq
Here we assume $\Gg_0>0$ and $\Gg_0 \geq 2|\Gg_2|$ to ensure the non-negativity of $\Gg$.
Then systems \eqref{phi1} and \eqref{phi2} can be simplified as
\begin{equation}
\begin{cases}
\ds (\Gg_0+{\Gg_2}+1)  \Gvf'_{1,1} +\Gg_2 \Gvf'_{1,3}
=-1,\\
(\Gg_0-{\Gg_2}+1)  \Gvf''_{1,1} +\Gg_2 \Gvf''_{1,3}
=0,\\
%2\Gg_2 \Gvf'_{1,2}+\Gg_0 \Gvf_{1,0}
%=0,\\
\Gg_2 \Gvf_{1,n+2}+({\Gg_0}+|n|) \Gvf_{1,n}+\Gg_2\Gvf_{1,n-2}
=0 \quad \mbox{if } n\neq  \pm 1,\\
\end{cases}
\label{phi21}
\end{equation}
and
\begin{equation}
\begin{cases}
(\Gg_0+{\Gg_2}+1)  \Gvf'_{2,1} +\Gg_2   \Gvf'_{2,3}
=0,\\
\ds (\Gg_0-{\Gg_2}+1) \Gvf''_{2,1} +\Gg_2  \Gvf''_{2,3}
=1,\\
%2\Gg_2 \Gvf'_{2,2}+\Gg_0 \Gvf_{2,0}
%=0,\\
\Gg_2 \Gvf_{2,n+2}+({\Gg_0}+|n|) \Gvf_{2,n}+\Gg_2\Gvf_{2,n-2}
=0 \quad \mbox{if } n\neq  \pm 1.\\
\end{cases}
\label{phi22}
\end{equation}
Then the first and third identities of system \eqref{phi21} yield the following system of $\Gg$ and $\Gvf'_1$:
\begin{equation}
\begin{bmatrix}
\Gg_0+\Gg_2+1 & \Gg_2  \\
\Gg_2  & \Gg_0+3  & \Gg_2 &&\textbf{0}\\
&\ddots  & \ddots & \ddots  \\
&   & \Gg_2 & \Gg_0+2n-1 &    \Gg_2  \\
& \textbf{0}  & & \ddots & \ddots & \ddots\\
\end{bmatrix}\begin{bmatrix}
\Gvf'_{1,1} \\
\Gvf'_{1,3}\\
\vdots\\
\Gvf'_{1,2n-1}\\
\vdots\\
\end{bmatrix}=\begin{bmatrix}
-1\\
0\\
\vdots\\
0\\
\vdots \\
\end{bmatrix},
\label{gaphi1}
\end{equation}
while the second and third identities of system \eqref{phi21} yield the system of $\Gg$ and $\Gvf''_1$:
\begin{equation}
\begin{bmatrix}
\Gg_0-\Gg_2+1 & \Gg_2  \\
\Gg_2  & \Gg_0+3  & \Gg_2 &&\textbf{0}\\
&\ddots  & \ddots & \ddots  \\
&   & \Gg_2 & \Gg_0+2n-1 &    \Gg_2  \\
& \textbf{0}  & & \ddots & \ddots & \ddots\\
\end{bmatrix}\begin{bmatrix}
\Gvf''_{1,1} \\
\Gvf''_{1,3}\\
\vdots\\
\Gvf''_{1,2n-1}\\
\vdots\\
\end{bmatrix}=\begin{bmatrix}
0\\
0\\
\vdots\\
0\\
\vdots \\
\end{bmatrix}.
\label{gaphi2}
\end{equation}

Let, for $N\geq 3$,
\begin{equation*}
\mathcal{A}_N=\begin{bmatrix}
\Gg_0+\Gg_2+1 & \Gg_2  \\
\Gg_2  & \Gg_0+3 & \Gg_2 & & \textbf{0} \\
& \ddots &\ddots  &\ddots & & \\
& &\Gg_2  & \Gg_0+2N-3 & \Gg_2\\
&  \textbf{0}&& \Gg_2 & \Gg_0+2N-1  \\
\end{bmatrix}, 
\end{equation*}
and
\begin{equation*}
\mathcal{B}_N=\begin{bmatrix}
\Gg_0-\Gg_2+1 & \Gg_2  \\
\Gg_2  & \Gg_0+3 & \Gg_2 & & 0 \\
& \ddots &\ddots  &\ddots & & \\
& 0&\Gg_2  & \Gg_0+2N-3 & \Gg_2\\
&  && \Gg_2 & \Gg_0+2N-1  \\
\end{bmatrix}.
\end{equation*}
Due to the recursive relation in \eqref{phi21}, we have
\beq\label{N}
\Gg_2\Gvf_{1,2N-3}+({\Gg_0}+2N-1) \Gvf_{1,2N-1}=-\Gg_2 \Gvf_{1,2N+1}.
\eeq
Thus \eqref{gaphi1} can be rewritten as
\begin{equation}\label{Aphi}
\mathcal{A}_N\begin{bmatrix}
\Gvf'_{1,1} \\
\Gvf'_{1,3}\\
\vdots\\
\Gvf'_{1,2N-3}\\
\Gvf'_{1,2N-1}
\end{bmatrix}=\begin{bmatrix}
-1\\
0\\
\vdots\\
0\\
-\Gg_2 \Gvf'_{1,2N+1}
\end{bmatrix}.
\end{equation}

Similarly, \eqref{gaphi2} can be written as
\beq\label{Bphi}
\mathcal{B}_N\begin{bmatrix}
	\Gvf''_{1,1} \\
	\Gvf''_{1,3}\\
	\vdots\\
	\Gvf''_{1,2N-3}\\
	\Gvf''_{1,2N-1}
\end{bmatrix}=\begin{bmatrix}
	0\\
	0\\
	\vdots\\
	0\\
	-\Gg_2 \Gvf''_{1,2N+1}
\end{bmatrix}.
\eeq

Let $d_1=\Gg_0+\Gg_2+1$, $d_2=\Gg_0+3$, $\ldots$, and $d_N= \Gg_0+2N-1$ so that $\mathcal{A}_N$ takes the form
\begin{equation*}
\mathcal{A}_N=\begin{bmatrix}
d_1 & \Gg_2  \\
\Gg_2  & d_2 & \Gg_2 & & \textbf{0} \\
& \ddots &\ddots  &\ddots & & \\
& &\Gg_2  & d_{N-1} & \Gg_2\\
&  \textbf{0}&& \Gg_2 & d_N  \\
\end{bmatrix}.
\end{equation*}
By Theorem 2.1 in \cite{inverse}, the tridiagonal matrix $\mathcal{A}_N$ is invertible and the inverse $\mathcal{A}_N^{-1}=(a_{kj})_{1\leq k,j\leq N}$ is given by:
\begin{equation*}
\begin{split}
&a_{11}=\left(d_1-\frac{\Gg_2^2\tau_3}{\tau_2}\right)^{-1},\\
&a_{NN}=\left(d_N-\frac{\Gg_2^2\xi_{N-2}}{\xi_{N-1}}\right)^{-1},\\
&a_{kk}=\left(d_i-\frac{\Gg_2^2\xi_{k-2}}{\xi_{k-1}}-\frac{\Gg_2^2\tau_{k+2}}{\tau_{k+1}}\right)^{-1}, \quad k=2,3,\dots,N-1,\\
&a_{kj}=\ds
\begin{cases}
(-\Gg_2)^{N-1}\d\frac{\xi_{k-1}}{\xi_{j-1}}a_{jj} \quad \mbox{if } k<j,\\
(-\Gg_2)^{N-1}\d\frac{\tau_{k+1}}{\tau_{j+1}}a_{jj} \quad \mbox{if } k>j,
\end{cases}
\end{split}
\end{equation*}
where
\begin{equation*}
\xi_k=\begin{cases}
1 \quad &\mbox{if } k=0,\\
d_1 \quad &\mbox{if } k=1,\\
d_k\xi_{k-1}-\Gg_2^2\xi_{k-2} \quad &\mbox{if }k=2,3,\dots,N,
\end{cases}
\end{equation*}
and
\begin{equation*}
\tau_k=\begin{cases}
1 \quad &\mbox{if } k=N+1,\\
d_N \quad &\mbox{if } k=N,\\
d_k\tau_{k+1}-\Gg_2^2\tau_{k+2} \quad &\mbox{if } k=N-1,N-2,\dots,1.
\end{cases}
\end{equation*}

Then it is easy to solve \eqref{Aphi} and obtain
\beq\label{phi1re}
\Gvf'_{1,2k-1}=-a_{k1}-\Gg_2 \Gvf'_{1,2N+1}a_{kN}, \quad k=1,\dots,N.
\eeq
Similarly, we can solve \eqref{Bphi} and obtain
$$
\Gvf''_{1,2k-1}=-\Gg_2 \Gvf''_{1,2N+1}b_{kN}, \quad k=1,\dots,N,
$$
where $(b_{ij})$ is the inverse matrix of $\mathcal{B}_N$.

Since $\Gvf_1\in H^{1/2}(\p\GD)$, we have $\lim_{N\rightarrow \infty} \Gvf_{1,2N+1}=0$. Then it is easy to see that as $N\rightarrow \infty$ we have
\beq
\Gvf'_{1,1} =-1/(\Gg_0+\Gg_2+1),
\label{Rephi11}
\eeq
and
\beq\label{Imphi1}
\Gvf''_{1,2n-1}=0,\quad n=0,\pm 1,\pm 2,\dots.
\eeq

Similarly, it is easy to see from system \eqref{phi22} that
\begin{equation*}
\mathcal{A}_N\begin{bmatrix}
\Gvf'_{2,1} \\
\Gvf'_{2,3}\\
\vdots\\
\Gvf'_{2,2N-3}\\
\Gvf'_{2,2N-1}
\end{bmatrix}=\begin{bmatrix}
0\\
0\\
\vdots\\
0\\
-\Gg_2 \Gvf'_{2,2N+1}
\end{bmatrix},
\quad
\mathcal{B}_N\begin{bmatrix}
\Gvf''_{2,1} \\
\Gvf''_{2,3}\\
\vdots\\
\Gvf''_{2,2N-3}\\
\Gvf''_{2,2N-1}
\end{bmatrix}=\begin{bmatrix}
1\\
0\\
\vdots\\
0\\
-\Gg_2 \Gvf''_{2,2N+1}
\end{bmatrix}.
\end{equation*}

In the same way as before, it is easy to solve the above systems and obtain
$$
\Gvf'_{2,2k-1}=-\Gg_2 \Gvf'_{1,2N+1}a_{kN}, \quad k=1,\dots,N,
$$
and
\beq\label{phi2im}
\Gvf''_{2,2k-1}=1-\Gg_2 \Gvf''_{2,2N+1}b_{kN}, \quad k=1,\dots,N.
\eeq
Hence
\beq\label{Rephi2}
\quad \Gvf'_{2,2n-1}=0,\quad n=0,\pm 1,\pm 2,\dots,
\eeq
and
\beq
\Gvf''_{2,1} =1/(\Gg_0-\Gg_2+1).
\label{Imphi21}
\eeq

To determine the coefficients $\Gvf_{j,2n}$, $n=0,\pm 1,\pm 2,\dots$, for $j=1,2$, from the recursive relations in \eqref{phi21} and \eqref{phi22}, we have
$$\Gg_2 \Gvf_{j,2} + \Gg_0 \Gvf_{j,0} + \Gg_2\Gvf_{j,-2}=0.$$
Since $\Gg_0>0$, it is easy to see that $\Gvf_{j,0}=0$.
Using the previous technique, it is easy to obtain
$$
\Gvf_{j,2n}=0, \quad n=0,\pm 1,\pm 2,\dots, \quad \mbox{for } j=1,2.
$$
Then the two density functions $\Gvf_j$, $j=1,2$, have the following forms
\begin{equation*}
\begin{cases}
\ds \Gvf_{1}=2\sum_{n=1}^{+\infty} \Gvf'_{1,2n-1} \cos(2n-1)\Gt,\\
\ds \Gvf_{2}=-2\sum_{n=1}^{+\infty} \Gvf''_{2,2n-1} \sin(2n-1)\Gt,\\
\end{cases}
\end{equation*}
where $ \Gvf'_{1,1} $ and $ \Gvf''_{2,1} $ are given by \eqref{Rephi11} and \eqref{Imphi21}, $\Gvf'_{1,2n-1}$ and $\Gvf''_{2,2n-1}$ are determined by \eqref{phi1re} and \eqref{phi2im}.

Recall that we try to seek $\Gg=\Gg_0+2\Gg_2\cos2\Gt$ with unknowns $\Gg_0$ and $\Gg_2$ so that $(\GD,\Gg)$ has the polarization tensor in the form of \eqref{Tgamma}. By comparing \eqref{Tgamma} and \eqref{Tij}, there should satisfy
\begin{equation}
\begin{cases}
T_{11}&=\ds 2\pi \left((\Gg_0+\Gg_2-1) \Gvf'_{1,1} +\Gg_2 \Gvf'_{1,3}\right)=2\pi|b_\GO|,\\
T_{12}&=\ds 2\pi \left((\Gg_0+\Gg_2-1) \Gvf'_{2,1} +\Gg_2 \Gvf'_{2,3}\right)=0,\\
T_{21}&=\ds 2\pi\left((-\Gg_0+\Gg_2+1) \Gvf''_{1,1} -\Gg_2 \Gvf''_{1,3}\right)=0,\\
T_{22}&=\ds 2\pi \left((-\Gg_0+\Gg_2+1) \Gvf''_{2,1} -\Gg_2 \Gvf''_{2,3}\right)=-2\pi|b_\GO|.
\end{cases}
\label{T1221}
\end{equation}
It is easy to see $T_{12}=0$, $T_{21}=0$ from \eqref{Imphi1} and \eqref{Rephi2}.

From the first identity in \eqref{phi21} we have
\beq\label{phi13}
\Gg_2 \Gvf'_{1,3}=-1-(\Gg_0+\Gg_2+1) \Gvf'_{1,1}.
\eeq
Plugging \eqref{phi13} into the first identity of \eqref{T1221} and by \eqref{Rephi11} we have
\beq\label{gamma12}
\frac{2}{\Gg_0+\Gg_2+1}-1=|b_\GO|.
\eeq

Similarly, from the second identity in \eqref{phi22} we have
\beq\label{phi23}
\Gg_2 \Gvf''_{2,3}=1-(\Gg_0-\Gg_2+1) \Gvf''_{2,1}.
\eeq
Plugging \eqref{phi23} into the last identity of \eqref{T1221} and by \eqref{Imphi21} we have
\beq\label{gamma21}
\frac{2}{\Gg_0-\Gg_2+1}-1=-|b_\GO|.
\eeq

The coefficients $\Gg_0$ and $\Gg_2$ can be easily determined by solving \eqref{gamma12} and \eqref{gamma21}. And we obtain that
\begin{equation*}
\begin{cases}
\ds \Gg_0=\frac{1}{1+|b_\GO|}+\frac{1}{1-|b_\GO|}-1,\\
\ds \Gg_2=\frac{1}{1+|b_\GO|}-\frac{1}{1-|b_\GO|}.
\end{cases}
\end{equation*}
This proves Proposition \ref{main_result}.

\section{Numerical examples}
In this section we give some numerical examples. Let $\GO$ be an ellipse of major axis $a$ and minor axis $b$. A conformal mapping from the exterior of unit disk $\GD$ to the exterior of $\GO$ is
\begin{equation*}
z=\Phi(\Gz)=\frac{a+b}{2}\left(\Gz+\frac{a-b}{a+b}\frac{1}{\Gz}\right), \quad \Gz\in \Cbb\setminus \GD, \quad z\in \Cbb\setminus \Omega.
\end{equation*}
Scaling domain $\GO$ by $\frac{2}{a+b}$ (still denote the scaled domain by $\GO$), we have
\beq
\Phi_\GO(\Gz)=\Gz+\frac{a-b}{a+b}\frac{1}{\Gz}.
\label{conf1}
\eeq
Then $b_\GO=\frac{a-b}{a+b}$. Then the interface parameter $\Gg$ on $\p \GD$ in \eqnref{Ggdisk} is given by
\begin{equation*}
\begin{split}
\Gg(e^{i\Gt})&=\frac{1}{1+|b_\GO|}+\frac{1}{1-|b_\GO|}-1+\left(\frac{2}{1+|b_\GO|}-\frac{2}{1-|b_\GO|}\right)\cos2\Gt\\
&=\left(\frac{a+b}{2a}+\frac{a+b}{2b}-1\right)+\left(\frac{a+b}{a}-\frac{a+b}{b}\right)\cos2\Gt.
\end{split}
\end{equation*}
Transform the interface parameter back to $z$-plane through conformal mapping \eqref{conf1}, we have from \eqref{Gbcon} that
\begin{equation*}
\Gb(\Phi_\GO(e^{i\Gt}))=\left(\left(\frac{a+b}{2a}+\frac{a+b}{2b}-1\right)+\left(\frac{a+b}{a}-\frac{a+b}{b}\right)\cos2\Gt\right)\left|1-\frac{a-b}{a+b}\frac{1}{e^{2i\Gt}}\right|^{-1}.
\end{equation*}

For example, let $\GO$ be an ellipse of major axis $a=5/4$ and minor axis $b=3/4$. Then $|b_\GO|=1/4$. The interface parameter
\beq
\Gb(\Phi_\GO(e^{i\Gt}))=\left(\frac{17}{15}-\frac{16}{15}\cos2\Gt\right)\frac{1}{\left|1-1/(4e^{2i\Gt})\right|}.
\label{betaellipse}
\eeq

Using finite element method, we solve \eqref{main} . We also solve the perfect bonding problem which is 
\beq\label{perfect}
\begin{cases}
	\Delta u=0 \quad &\mbox{in }\Rbb^2\setminus\ol{\GO},\\
	u=\lambda \quad&\mbox{on }\p \GO,\\
	\vspace{0.3em}
	u(x)-a \cdot x=O(|x|^{-1}) \quad&\mbox{as }|x|\rightarrow\infty,
\end{cases}
\eeq
where the constant $\lambda$ is determined by the condition $\int_{\p\GO}\frac{\p u}{\p \nu}|_+ \, ds=0$. Figure \ref{ellipse} show computational results. The left column is when the uniform field is parallel to $x_1$-axis and the right column is $x_2$-axis. Top row is for perfectly bonding interfaces and the bottom row is for imperfect interfaces with the parameter $\Gb$ defined by \eqnref{betaellipse}. Figures show clearly that the field perturbation is much weaker for imperfect interfaces.

\begin{figure}[h]
	\centering
	\begin{subfigure}{0.4\textwidth}
		\includegraphics[scale=0.33]{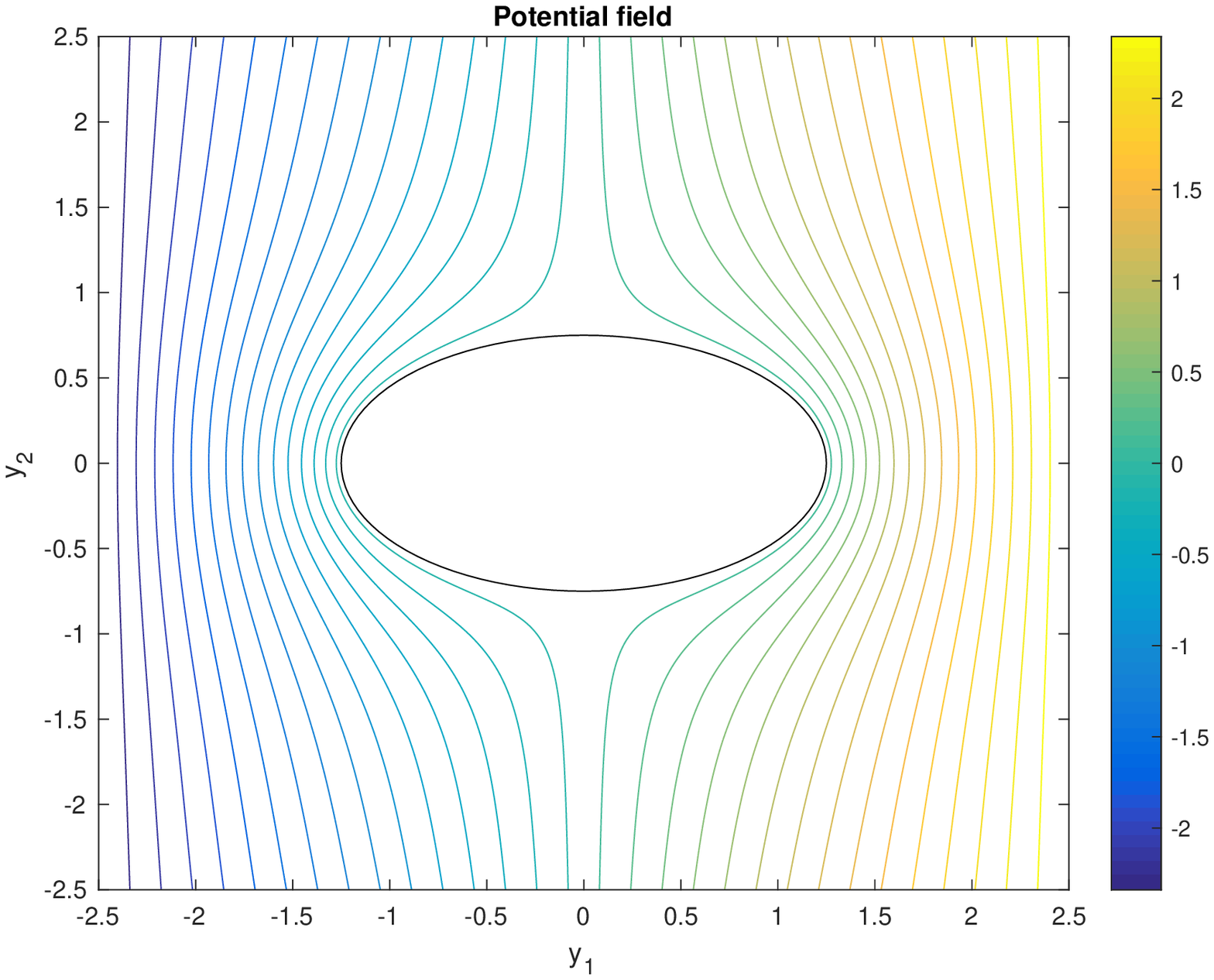}
		\subcaption{}
	\end{subfigure}
	\begin{subfigure}{0.4\textwidth}
		\includegraphics[scale=0.33]{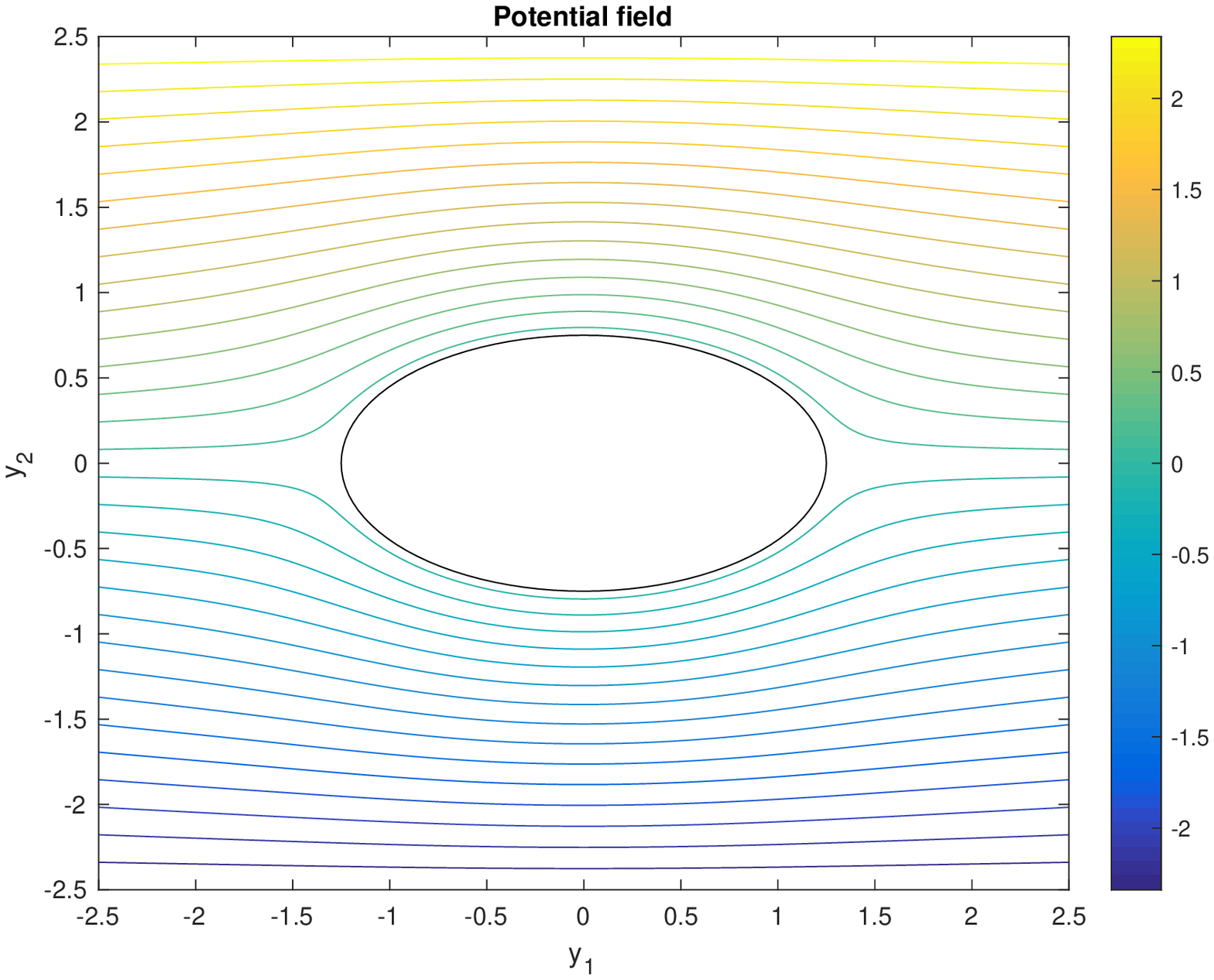}
		\caption{}
	\end{subfigure}\\
	\begin{subfigure}{0.4\textwidth}
		\includegraphics[scale=0.33]{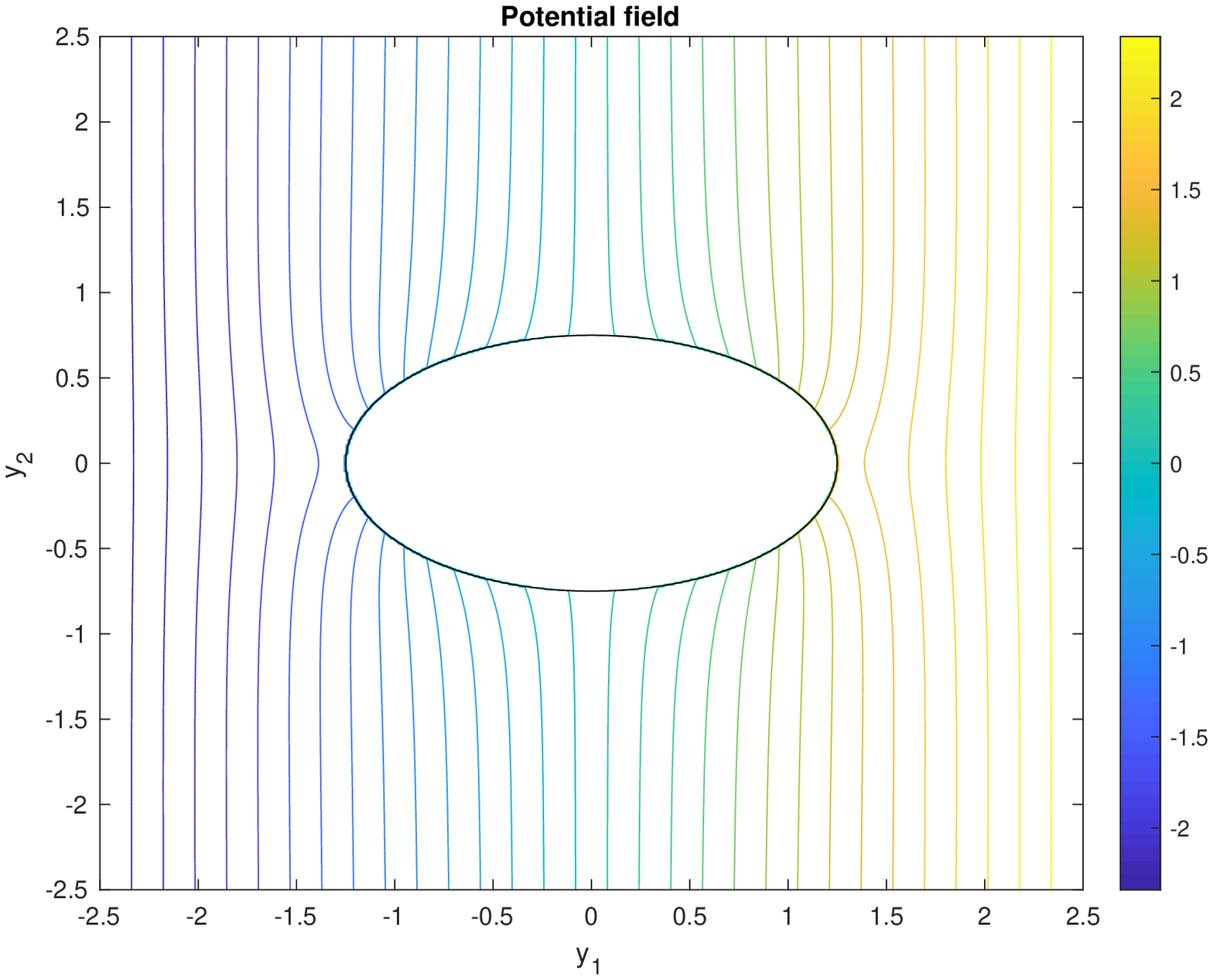}
		\caption{}
	\end{subfigure}
	\begin{subfigure}{0.4\textwidth}
		\includegraphics[scale=0.33]{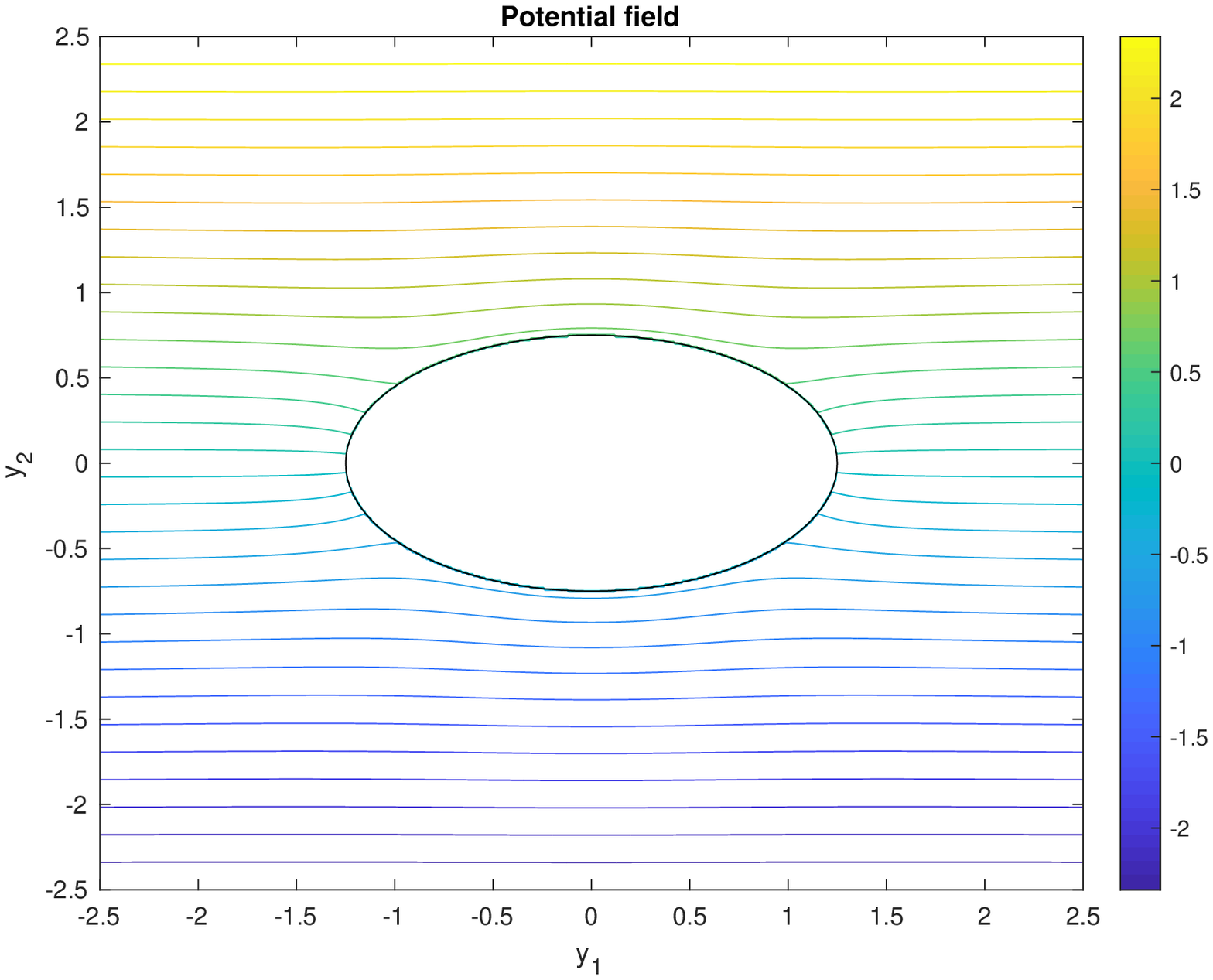}
		\caption{}
	\end{subfigure}
	\caption{(a) Perfectly bonding interface, $a=(1,0)$; (b) Perfectly bonding interface, $a=(0,1)$; (c) Imperfect interface, $a=(1,0)$; (d) Imperfect interface, $a=(0,1)$. The figures show that the field perturbation is much weaker for the imperfect interface.}
	\label{ellipse}
\end{figure}

We consider another example: a droplet-shaped domain. In this case, $\Phi_\GO$ is given by
\beq
\Phi_\GO(\Gz)=\Gz+\frac{1}{4\Gz}-\frac{1}{8\Gz^2}+... \quad \mbox{as } |\Gz|\rightarrow\infty,
\eeq
then $|b_\GO|=1/4$.
Thus the interface parameter is given by
\beq
\Gb(\Phi_\GO(e^{i\Gt}))=\left(\frac{17}{15}-\frac{16}{15}\cos 2\Gt\right)\left|\frac{1+e^{-i\Gt}}{(1+1/(2e^{i\Gt}))^2}\right|^{-1}.
\label{betadroplet}
\eeq

Results are shown in Figure \ref{droplet}. Figures again show clearly that the field perturbation is much weaker for imperfect interfaces.
\begin{figure}[h]
	\centering
	\begin{subfigure}{0.4\textwidth}
		\includegraphics[scale=0.33]{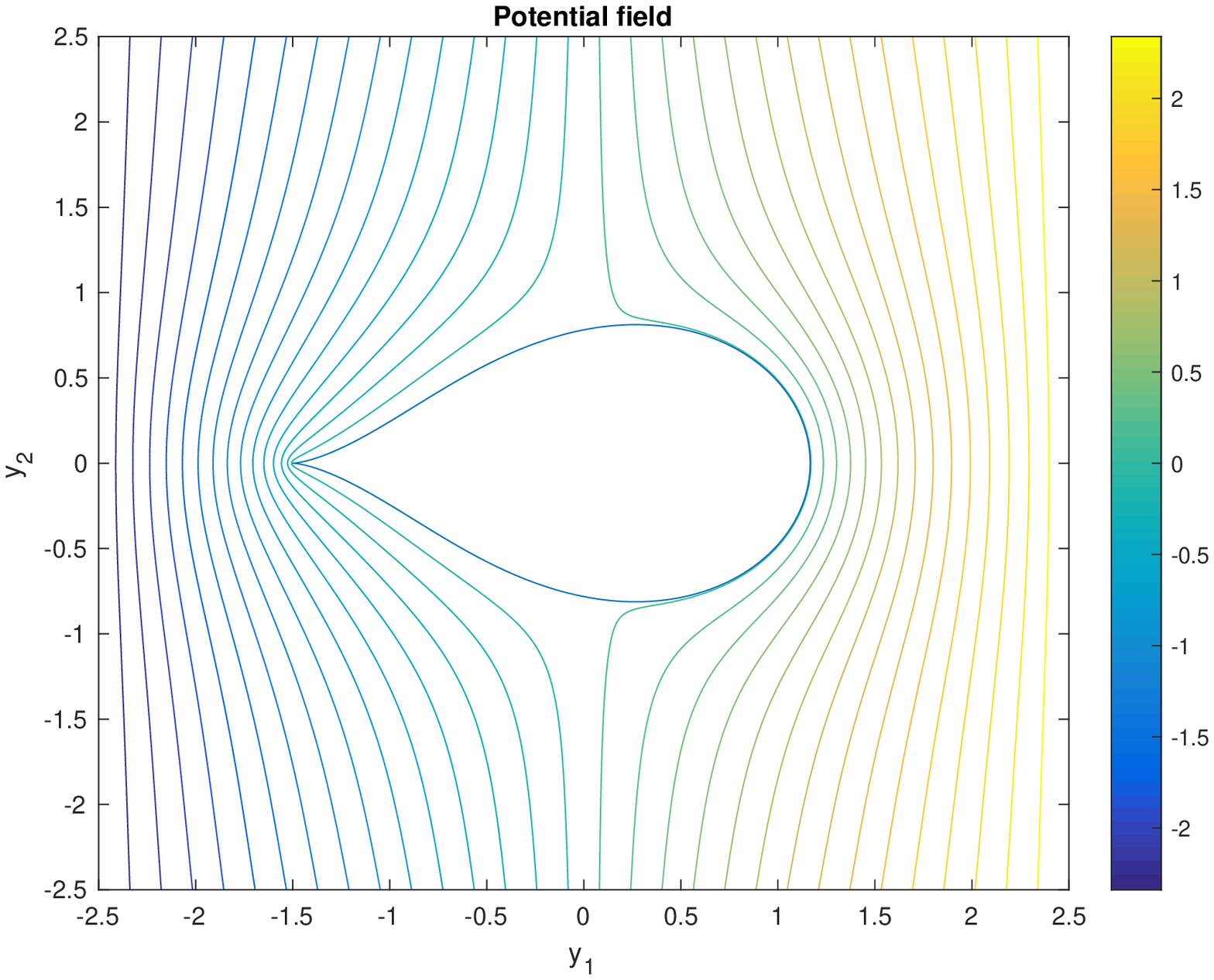}
		\subcaption{}
	\end{subfigure}
	\begin{subfigure}{0.4\textwidth}
		\includegraphics[scale=0.33]{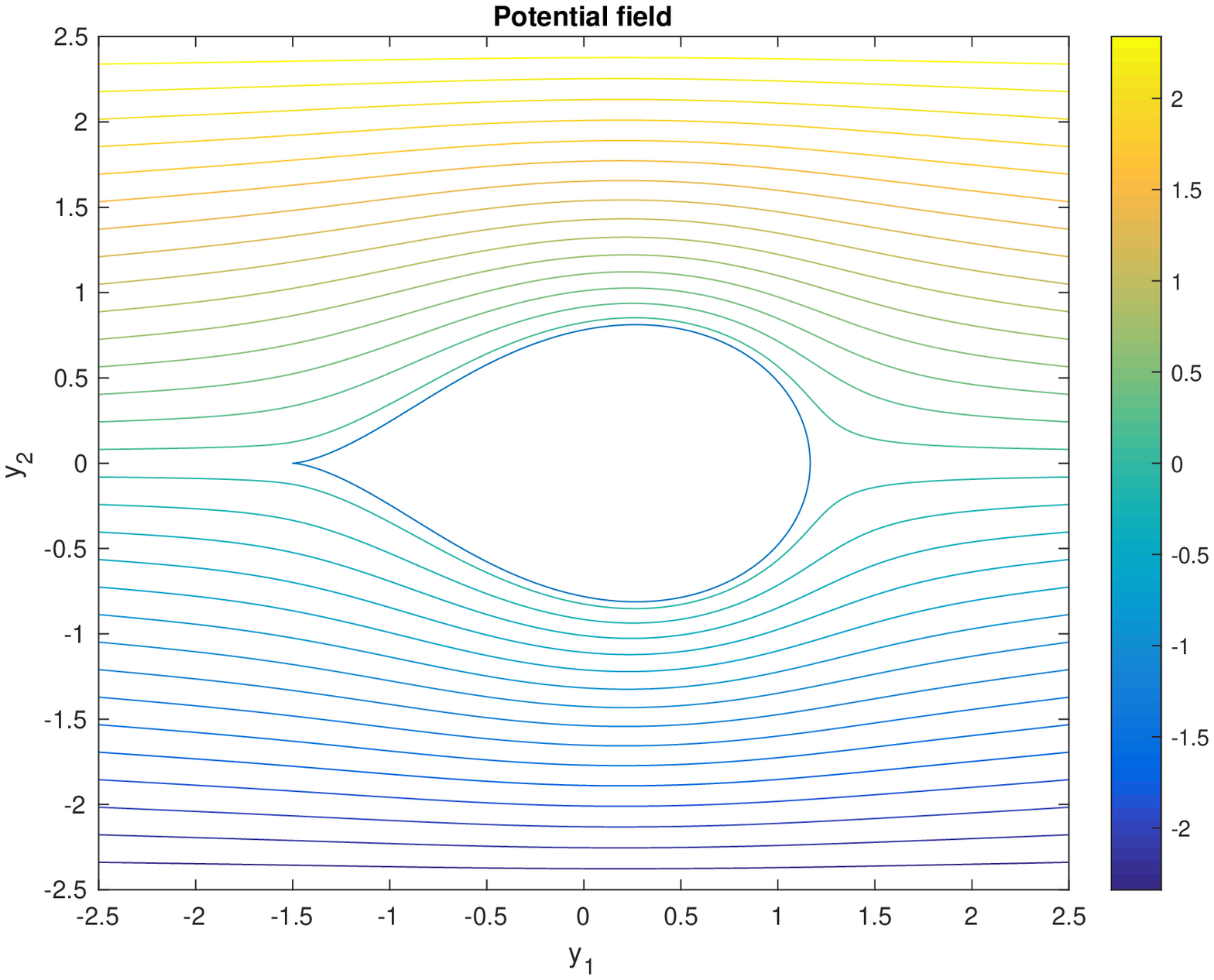}
		\caption{}
	\end{subfigure}\\
	\begin{subfigure}{0.4\textwidth}
		\includegraphics[scale=0.33]{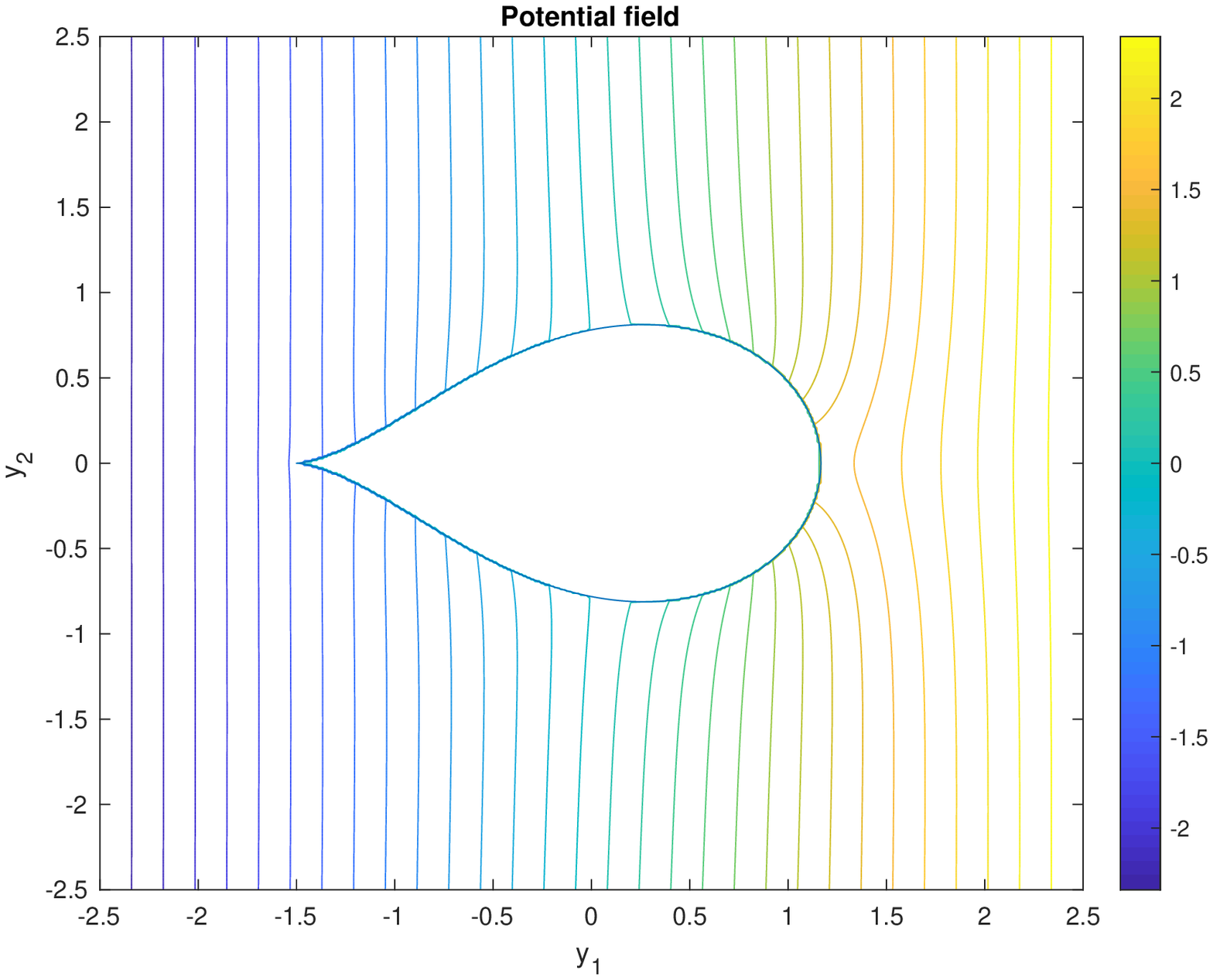}
		\caption{}
	\end{subfigure}
	\begin{subfigure}{0.4\textwidth}
		\includegraphics[scale=0.33]{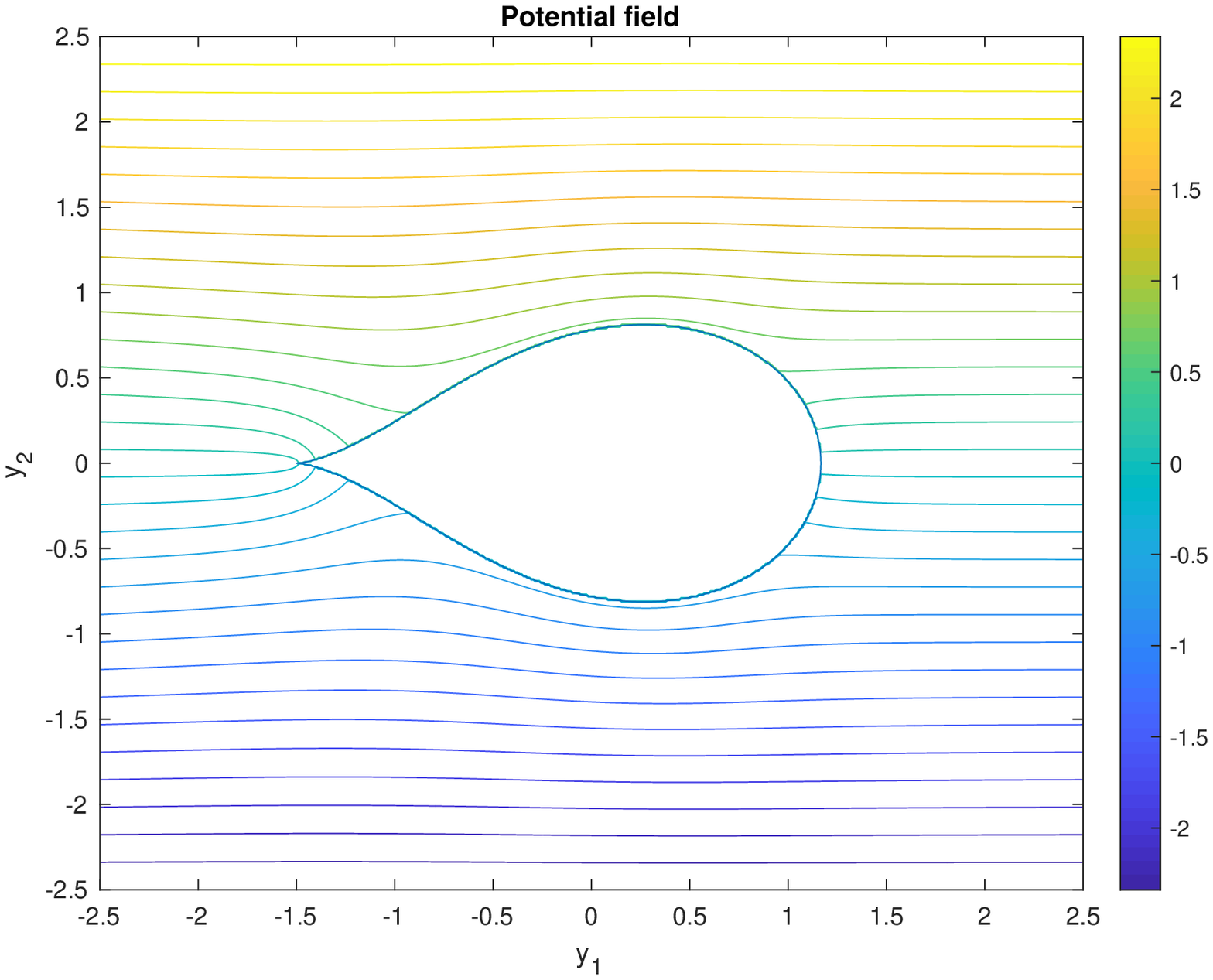}
		\caption{}
	\end{subfigure}
	\caption{(a) Perfectly bonding interface, $a=(1,0)$; (b) Perfectly bonding interface, $a=(0,1)$; (c) Imperfect interface, $a=(1,0)$; (d) Imperfect interface, $a=(0,1)$. The figures show that the field perturbation is much weaker for the imperfect interface.}
	\label{droplet}
\end{figure}

%%%%%%%%%%%%%%%%%%%%%%%%%%%%%%%%
\section{Conclusion}
%%%%%%%%%%%%%%%%%%%%%%%%%%%%%%%%

We introduce a new concept of weakly neutral inclusions, or polarization tensor vanishing structures, which perturb, upon insertion, the uniform fields only weakly. We show that inclusions of arbitrary shape can be realized as weakly neutral inclusions by introducing proper imperfect interface parameter, provided that the domain satisfies a certain condition (the condition \eqnref{assumebGO}). To construct the interface parameter $\Gb$, we pull back the problem to the unit disk using a conformal mapping and then use only $1$ and $\cos 2\Gt$. It is likely that this restriction \eqnref{assumebGO} is removed if we use more terms like $\cos 4\Gt$, etc in construction of $\Gb$. To the best of our knowledge, this is the first result on inclusions of general shape weakly neutral to multiple uniform fields. We also proved that the only shape with the imperfect interface which is neutral to all uniform fields is an ellipse in two dimensions and an ellipsoid in three dimensions.
 
This paper studies the neutral inclusions with imperfect interfaces. Recently authors with S. Sakaguchi also prove that there are inclusions of general shape with coated structure, namely, core-shell structure, which are weakly neutral to multiple uniform fields. This result is reported in \cite{KLS}.

%%%%%%%%%%%%%%%%%%%%%%%%%%%%%%%%%%%%%%%%%%%%
\section{Appendix: A geometric characterization of ellipsoids}
%%%%%%%%%%%%%%%%%%%%%%%%%%%%%%%%%%%%%%%%%%%%

We consider the problem \eqnref{LC} when $\sigma_m$ is anisotropic and $\Gs_m - \Gs_c I$ is either positive-definite or negative-definite. We prove the following theorem.

\begin{thm}\label{thm:ellipsoid}
Let $\GO$ be a simply connected domain in $\Rbb^d$ with $C^1$ boundary. Suppose that $\Gs_m$ appearing in \eqnref{LC} is anisotropic. Then $\GO$ with a certain interface parameter $\Gb$ is neutral to all uniform fields if and only if $\GO$ is an ellipse if $d=2$ and ellipsoid if $d=3$. If $\Gs_m$ is isotropic, then it must be a disk or a ball.
\end{thm}

\begin{proof}
We only prove the case when $d=3$ since the proof for the two dimensional case is similar and simpler. Suppose that
$$
\Gs_m= \mbox{diag} [\Gs_m^1, \Gs_m^2, \Gs_m^3 ].
$$
Let $\Be_1=(1,0,0)$, $\Be_2=(0,1,0)$ and $\Be_3=(0,0,1)$.
	
According to the neutrality criteria in \cite{B.M}, if $\GO$ is neutral to one field $\Be_j$ for some $j$, then after rotating and translating $\GO$ if necessary, the interface parameter $\Gb$ should be chosen as
\beq
\Gb(x)=a_j  \frac{\nu_j(x)}{x_j},\quad x=(x_1,x_2,x_3)\in\p\GO,
\eeq
where $\nu_j$ is the $j$th component of the outward unit normal vector $\nu$ and
$$
a_j: =\frac{\Gs_c\Gs_m^j}{\Gs_c-\Gs_m^j}.
$$
So, $\GO$ with the interface parameter $\Gb$ is neutral to $\Be_j$ for all $j$ if and only if
\beq\label{condition}
a_1\frac{\nu_1}{x_1}=a_2\frac{\nu_2}{x_2}=a_3\frac{\nu_3}{x_3} \quad \mbox{on } \p\GO,
\eeq
where we may assume that $a_j >0$ for $j=1,2,3$. Thus Theorem \ref{thm:ellipsoid} is a consequence of the following theorem which provides a geometric characterization of ellipsoids. \qed

\begin{thm}[Characterization of ellipsoids]
Let $\GO$ be a simply connected domain in $\Rbb^3$ with $C^1$ boundary. Then, $\GO$ is an ellipsoid if and only if \eqnref{condition} holds. A similar characterization holds for ellipses.
\end{thm}	

\pf
If $\GO$ is an ellipsoid of the form $\sum_{j=1}^3 x_j^2/c_j^2$, then the outward normal vector on $\p\GO$ is given by
$$
\nu(x)= \left(\frac{x_1}{c_1^2 w(x)}, \frac{x_2}{c_2^2 w(x)},\frac{x_3}{c_3^2 w(x)}\right),
\quad
w(x) = \sqrt{\frac{x_1^2}{c_1^4}+\frac{x_2^2}{c_2^4}+\frac{x_3^2}{c_3^4}}.
$$
Thus \eqnref{condition} holds.

To prove the converse, let $U$ be an open set in $\Rbb^2$ and let
$\Phi:U\rightarrow \p\GO$ be a local chart. Then \eqref{condition} implies that $(a_2a_3\Phi_1, a_1a_3\Phi_2, a_1a_2\Phi_3)$ is proportional to the normal vector $\nu$ at every point of $\p \GO$. Therefore, since $\Phi_u$ is tangential to $\p\GO$, we have
$$
\frac{\p}{\p u}(a_2a_3\Phi^2_1+a_1a_3\Phi^2_2+a_1a_2\Phi^2_3)=2(\Phi_{1u}, \Phi_{2u}, \Phi_{3u}) \cdot (a_2a_3\Phi_1, a_1a_3\Phi_2, a_1a_2\Phi_3)=0.
$$
Likewise, we also have
$$
\frac{\p}{\p v}(a_2a_3\Phi^2_1+a_1a_3\Phi^2_2+a_1a_2\Phi^2_3)= 0.
$$
Thus we have
$$
a_2a_3\Phi^2_1+a_1a_3\Phi^2_2+a_1a_2\Phi^2_3=\mbox{constant}.
$$
It means that every local chart of $\p\GO$ is a part of ellipsoid
\beq\label{ell}
a_2a_3x_1^2+a_1a_3x_2^2+a_1a_2x_3^2=\mbox{constant}.
\eeq
One can easily show that the constants are the same by continuity argument. Therefore, $\p\GO$ is an ellipsoid.
\end{proof}

\end{document}